\documentclass[10pt]{amsart}
\numberwithin{equation}{section}
\usepackage{latexsym}
\usepackage{amsmath}
\usepackage{amssymb}
\usepackage{mathrsfs}
\usepackage{graphicx}
\usepackage{color}
\usepackage{pgfpages}
\usepackage{ifthen}
\usepackage{leftidx,tensor}
\usepackage[T1]{fontenc}
\usepackage[latin1]{inputenc}
\usepackage{mathtools}
\usepackage{comment}

\numberwithin{equation}{section}

\newtheorem{defi}{Definition}[section]
\newtheorem{theorem}[defi]{Theorem}
\newtheorem{lemma}[defi]{Lemma}

\newtheorem{proposition}[defi]{Proposition}
\newtheorem{remark}[defi]{Remark}
\newtheorem{remarks}[defi]{Remarks}

\newcommand{\be}{\begin{equation} \label}
\newcommand{\ee}{\end{equation}}
\newcommand{\bea}{\begin{eqnarray}\label}
\newcommand{\eea}{\end{eqnarray}}
\newcommand{\bas}{\begin{eqnarray*}}
\newcommand{\eas}{\end{eqnarray*}}
\newcommand{\bit}{\begin{itemize}}
\newcommand{\eit}{\end{itemize}}
\newcommand{\R}{\mathbb{R}}

\newcommand{\C}{\mathbb{C}}

\newcommand{\cA}{\mathcal{A}}

\newcommand{\E}{\mathbb{E}}

\newcommand{\eps}{\varepsilon}

\newcommand{\cE}{{\mathcal E}}

\newcommand{\cL}{{\mathcal L}}
\newcommand{\cB}{{\mathcal B}}

\newcommand{\ccor}{c_{\mbox{\tiny{cor}}}}
\newcommand{\Catm}{C_{\mbox{\tiny{atm}}}}
\newcommand{\Cocean}{C_{\mbox{\tiny{ocean}}}}
\newcommand{\Uatm}{U_{\mbox{\tiny{atm}}}}
\newcommand{\Uocean}{U_{\mbox{\tiny{ocean}}}}
\newcommand{\Ratm}{R_{\mbox{\tiny{atm}}}}
\newcommand{\Rocean}{R_{\mbox{\tiny{ocean}}}}
\newcommand{\ratm}{\rho_{\mbox{\tiny{atm}}}}
\newcommand{\rice}{\rho_{\mbox{\tiny{ice}}}}
\newcommand{\rocean}{\rho_{\mbox{\tiny{ocean}}}}
\newcommand{\tatm}{\tau_{\mbox{\tiny{atm}}}}
\newcommand{\tocean}{\tau_{\mbox{\tiny{ocean}}}}

\DeclareMathOperator*{\divergence}{div}

\DeclareMathOperator*{\trace}{tr}

\DeclareMathOperator{\Ret}{Re} 
\DeclareMathOperator{\Imt}{Im} 
\allowdisplaybreaks
\begin{document}

\title[Rigorous Analysis of  Hibler's Sea Ice Model]
{Rigorous Analysis and Dynamics of Hibler's Sea Ice Model}

\author{Felix Brandt}
\address{Technische Universit\"at Darmstadt\\
        Fachbereich Mathematik\\
        Schlossgartenstrasse 7\\
        64289 Darmstadt, Germany}
\email{felix.brandt@stud.tu-darmstadt.de}

\author{Karoline Disser}
\address{Universit\"at Kassel\\
        Institut f\"ur Mathematik \\
        Heinrich-Plett-Stra{\ss}e 40 \\
        34132 Kassel, Germany}
\email{kdisser@mathematik.uni-kassel.de}

\author{Robert Haller-Dintelmann}
\address{Technische Universit\"at Darmstadt\\
        Fachbereich Mathematik\\
        Schlossgartenstrasse 7\\
        64289 Darmstadt, Germany}
\email{haller@mathematik.tu-darmstadt.de}

\author{Matthias Hieber}
\address{Technische Universit\"at Darmstadt\\
        Fachbereich Mathematik\\
        Schlossgartenstrasse 7\\
        64289 Darmstadt, Germany}
\email{hieber@mathematik.tu-darmstadt.de}


\begin{abstract}
This article develops  for the first time a rigorous analysis of Hibler's model of sea ice dynamics. Identifying Hibler's ice stress as a quasilinear second order operator and regarding 
Hibler's model as a quasilinear evolution equation, it is shown  that Hibler's coupled sea ice model, i.e., the model coupling velocity, thickness and compactness of sea ice,  is locally 
strongly well-posed within the $L_q$-setting and also globally strongly well-posed for initial data close to constant  equilibria.         
\end{abstract}

\maketitle

\section{Introduction}\label{secintro}
Sea ice is a material with a complex mechanical and thermodynamical behaviour. Freezing sea water forms a composite of pure ice, liquid brine, air pockets and solid salt. The details 
of this formation depend on the laminar or turbulent environmental conditions, see e.g. \cite{Hib79}, \cite{Fel08} and \cite{Gol15}. This composite responds differently to heating, pressure or 
mechanical forces than for example the (salt-free) glacial ice of ice sheets.

The evolution of sea ice has attracted much attention in climate science due to its role as a hot spot in global warming. The state of the art concerning the  {\it modeling} of sea ice is described in the 
very recent survey article \cite{survey20} in the Notices of the American Mathematical Society.    

Somewhat surprisingly, the field of sea ice dynamics forms a terra incognita to rigorous mathematical analysis. In contrast to atmospheric or oceanic models, see e.g., the work of Lions, 
Temam, Wang \cite{LTW92a,LTW92b} and Cao and Titi \cite{CT07} for the primitive equations as well as the work of Majda \cite{Maj03} and Benacchio and Klein for atmospheric flows \cite{BK19}, 
rigorous analysis of sea ice models is essentially non-existent.    

The governing equations of large-scale sea ice dynamics that form the basis of virtually all sea ice models in climate science were suggested in a seminal  paper by Hibler \cite{Hib79} in 1979. Sea ice is 
here modeled as a material with a very specific constitutive law based on  viscous-plastic rheology. This model has been investigated numerically  by various communities
(see e.g.\@ \cite{Mehalle21,MK21,Meh19,MR17,SK18,DWT15,KDL15,BT14,LT09}), but it seems it never was studied from a rigorous analytical point of view. In fact, even the existence of 
weak solutions to Hibler's sea ice model seems to be unknown until today.

Moreover, fundamental 
questions in this respect such as thermodynamical consistency  of the Hibler model  with the second law of thermodynamics,  as well as existence, uniqueness and regularity properties of 
solutions of this sea ice PDE system  seem to be open problems. Under distinct simplifications, certain submodels were, however, considered  in 
\cite{Gra99} and \cite{GLS13} within the context of hyperbolic systems. The authors postulate  ill-posedness of these simplified submodels. 

As stated above,  Hibler's sea ice model was already investigated numerically by many authors.   
All of these approaches are based on various regularizations of the underlying ice stress tensor. For example, the original viscous-plastic equations have been regularized by means of 
additional artificial elasticity, see  e.g., the work of Hunke-Dukowicz \cite{HD97}, in order to improve computational efficiency. This elastic-viscous-plastic approach has been 
implemented then in many sea ice and 
climate models. Let us emphasize that simulations of sea ice behaviour show a distinct discrepancy whether the original Hibler or the regularized sea ice PDEs are being used, see \cite{LD12}. For a thorough 
numerical study based on  a regularization of the orginal viscous-plastic ice stress, we refer to the work of Mehlmann, Richter and Korn \cite{MK21,MR17,Meh19}.     
 
In this article for the first time we rigorously prove the existence and uniqueness of a strong solution to   Hibler's sea ice model. Our approach is based on the theory of   
quasilinear parabolic evolution equations and a regularization of Hibler's original ice stress $\sigma$. This regularization has been used already in various numerical approaches by 
Mehlmann, Richter and Korn, see  \cite{Meh19}, \cite{MR17}, \cite{MK21}.  

A key point of our analysis is the understanding of the term $\divergence \sigma$ as a strongly  elliptic quasilinear operator $A$ within the $L_q$-setting.
We show that its linearization,  subject to  Dirichlet boundary conditions, fulfills the Lopatinksii-Shapiro condition yielding the maximal $L_q$-regularity property of the linearized Hibler operator. 
The latter property is then extended to the coupled system, described precisely  later on in \eqref{eq:cs},  consisting of the  momentum equation for the velocity $u$ and the two balance laws for the 
mean ice thickness $h$ and the ice compactness $a$. Regarding this model as a quasilinear evolution equation, we obtain strong well-posedness of the fully coupled system. For background information on 
linear and quasilinear evolution equations we refer e.g. to \cite{ABHN11,DHP03,DDHPV04,KW04,DHP07,PS16,HRS20}.   

In our first main result we prove the existence and uniqueness of a local {\em strong solution} to \eqref{eq:cs} for suitably chosen  initial data and show that this solution depends continuously on the data, 
exists on a maximal time interval and regularizes instantly in time. Secondly, we show that this solution extends uniquely to a global strong solution provided the initial data are close
to the equilibria $v_* = (0,h_*,a_*)$, where $h_*$ and $a_*$ denote constants and the external forces vanish. 

To put our result in perspective, note that  the existence of a  weak solution to 
Hibler's sea ice model \eqref{eq:cs} is  {\em not} known until today. It is also interesting to observe that Hibler's sea ice stress tensor is related to the stress tensor of certain non-Newtonian 
fluids, as described e.g.\@ in \cite{BP07}. It was shown recently by Burczak, Modena and Szekelyhidi in \cite{BMS20} that under certain assumptions weak solutions to these non-Newtonian fluid models 
are highly non unique.

This article is organized as follows: Section 2 presents Hibler's model as well as  our  main well-posedness results for this system. In Section 3 we rewrite Hibler's operator as a second order 
quasilinear operator, whose linearization will be investigated in Section 4. There we show that the linearization of Hibler's operator is a strongly elliptic operator within the $L_q$-setting and that 
this operator subject to Dirichlet boundary conditions satisfies the maximal $L_q$-regularity property. After a short section on functional analytic properties of Hibler's operator, in Section 6  we present 
the proof of our local well-posedness  result. Finally, the proof of our global well-posedness  result is given in Section 7.

\section{Hibler's Viscous-Plastic Sea Ice Model and Main Results}\label{sechiblerseaicemodel}
In 1979, W.D. Hibler \cite{Hib79} proposed a rheology model for sea ice dynamics, which has become since then the  standard sea ice dynamics model and serves until today as a  basis for 
many numerical studies in this field. Roughly speaking, pack ice consists of rigid plates which drift freely in open water or are closely packed together in areas of high ice concentration. 
Although individual ice floes may have very different sizes, pack ice may be considered as  a highly fractured two-dimensional continuum. 

The momentum balance in this model is given by the  two-dimensional  equation     
\begin{align}\label{eqmomentumbalance}
    m(\dot{u} + u \cdot \nabla u) = \divergence \sigma - m \ccor n \times u - m g \nabla H + \tatm + \tocean, 
\end{align}     
where ${u: (0,\infty) \times \R^2 \to \R^2}$ denotes the horizontal ice velocity and  $m$ the ice mass per unit area.  
Moreover, ${- m \ccor \, n \times u}$ represents the Coriolis force with Coriolis parameter $\ccor > 0$ and unit vector ${n : \R^2 \to \R^3}$ normal to the surface, while 
$- m g \nabla H$ describes the force arising from changing sea surface tilt with  sea surface dynamic height ${H: (0,\infty) \times \R^2 \to [0,\infty)}$ and gravity $g$.
The terms $\tatm$ and $\tocean$ describe  atmospheric wind and oceanic forces given by
\begin{align}
\tatm  &= \ratm \Catm \vert \Uatm \vert \Ratm \Uatm, \label{def:tau1}\\
\tocean &= \rocean \Cocean \vert \Uocean - u \vert \Rocean (\Uocean - u), \label{def:tau2}
\end{align}
where $\Uatm$ and $\Uocean$ denote the surface winds of the atmosphere and the surface current of the ocean, respectively.  Furthermore, $\Catm$ and $\Cocean$ are air and ocean drag 
coefficients, $\ratm$ and $\rocean$ denote the densities for air and sea water and $\Ratm$ and $\Rocean$ are rotation matrices acting on wind and current vectors. For results on  
fluids driven by wind forces, we refer to \cite{BS01}.

Following Hibler \cite{Hib79}, the viscous-plastic rheology is given by a constitutive law that relates the internal ice stress $\sigma$ and the deformation tensor  
${\eps = \eps(u) = \frac{1}{2}(\nabla u + \nabla u ^T)}$ through an internal ice pressure $P $ and nonlinear bulk and shear viscosities, $\zeta$ and $\eta$, 
such that the principal  components of the stress lie on an elliptical yield curve  with the ratio of major to minor axes $e$. This constitutive  law is given by 
 \begin{align}\label{eq:sigma}
    \sigma & = 2 \eta(\eps,P) \eps + [\zeta(\eps,P) - \eta(\eps,P)]\trace(\eps)I - \frac{P}{2}I. 
\end{align}
The pressure $P$ measures the ice strength, depending on the thickness $h$ and the ratio $a$ of thick ice per unit area, and is explicitly given by 
\begin{equation}\label{def:P}
P=P(h,a)=p^\ast h \exp{(-c(1-a))},
\end{equation}
where $p^\ast>0$ and $c>0$ are given constants. The bulk and shear viscosities $\zeta$ and $\eta$ increase with pressure and decreasing deformation tensor and are given by 
$$
\zeta(\eps,P)=\frac{P}{2 \triangle(\eps)}  \quad \mbox{and}  \quad \eta(\eps,P)=e^{-2} \zeta(\eps,P), 
$$
where 
$$
 \triangle^2(\eps) := \Bigl(\eps_{11}^2 + \eps_{22}^2\Bigr)\Bigl(1+\frac{1}{e^2}\Bigr) + \frac{4}{e^2} \eps_{12}^2 + 2 \eps_{11} \eps_{22}\Bigl(1- \frac{1}{e^2}\Bigr), 
$$
and $e$ as described above is the ratio of the long axis to the short axis of the elliptical  yield curve. 
The above law represents an idealized viscous-plastic material, whose viscosities, however, become singular if $\triangle$ tends to zero. 

For this reason, already Hibler proposed to regularize this behaviour  by bounding the viscosities when $\triangle$ is getting small and by defining maximum values $\zeta_{\mbox{\tiny max}}$ and 
$\eta_{\mbox{\tiny max}}$ for 
$\zeta$ and $\eta$. Then $\zeta$ and $\eta$ become 
$$
\zeta' = \min\{\zeta,\zeta_{\mbox{\tiny max}}\} \quad \mbox{and} \quad  \eta'= \min\{\eta,\eta_{\mbox{\tiny max}}\}.
$$ 
This formulation of the viscosities leads, however, to non smooth rheology terms. To enforce smoothness, several regularizations have been considered in the literature, see e.g. \cite{MR17}, 
\cite{LT09}. For example, Lemieux and Tremblay \cite{LT09} replaced $\zeta$ by $\zeta= \zeta_{\mbox{\tiny max}} \tanh (P/(2\triangle\zeta_{\mbox{\tiny max}})$. 

An elastic-viscous-plastic stress tensor was introduced  by Hunke and 
Dukowicz in \cite{HD97}. Starting from the observation that the relation \eqref{eq:sigma} for $\sigma$ can be rewritten as 
$\frac{1}{2\eta}\sigma + \frac{\eta-\zeta}{4 \eta\zeta} \trace\sigma + \frac{P}{4\zeta}I = \eps$, they proposed the relation 
$$
\frac{1}{E}\partial_t \sigma + \frac{1}{2\eta}\sigma + \frac{\eta-\zeta}{4 \eta\zeta} \trace\sigma + \frac{P}{4\zeta}I = \eps.
$$ 
Note hat the relation \eqref{eq:sigma} is obtained in the limit $E \to \infty$,  while for $\zeta,\eta \to \infty$ one recovers the elasticity equation  $\frac{1}{E}\partial_t \sigma = \eps$.  In 
the pure plastic case, the compressive stress $\sigma_d = \trace \sigma$ and the shear stress $\sigma_s= ( (\sigma_{11}-\sigma_{22})^2 + 4 \sigma_{12}^2)^{1/2}$ are linked by the relation  
$(\sigma_d + P)^2 + e^2\sigma_s^2 = P^2$, which leads to  an elliptical yield curve.

Following \cite{MK21}, see also \cite{KHLFG00}, we consider for $\delta>0$ the regularization  
$$
{\triangle_\delta (\eps) := \sqrt{\delta + \triangle^2 (\eps)}}.
$$
We then set $\zeta_\delta=\frac{P}{2 \triangle_\delta(\eps)}$ and $\eta_\delta=e^{-2} \zeta_\delta$ as well as 
 \begin{align}\label{defsigmadelta}
    \sigma_\delta := 2 \eta_\delta \eps + [\zeta_\delta - \eta_\delta]\trace(\eps)I - \frac{P}{2}I. 
\end{align} 
We consider in the following the above momentum equation \eqref{eqmomentumbalance}  in a bounded domain $\Omega \subset \R^2$ with boundary of class $C^2$. It is coupled to two balance equations for the 
mean ice thickness 
\begin{equation}\label{eq:hkappa}
h: J \times \Omega  \to [\kappa,\infty) \quad \mbox{for some} \quad \kappa >0, 
\end{equation}
and the ice compactness $a: J \times \Omega  \to \R$ with $a \geq \alpha$ for some $\alpha \in \R$ given by 
\begin{equation*}\label{eq:hA}
  \left\{ \begin{array}{ll}
	\dot{h} + \divergence(u h) &=S_h + d_h \Delta h, \\[2mm]
	\dot{a} + \divergence(u a) &=S_a + d_a \Delta a.  \\[2mm]
	\end{array} \right.
\end{equation*} 
Here, $\kappa$ is a small parameter that indicates the transition to open water in the sense that for $m = \rice h$  a value of $h(t,x,y)$ less than $\kappa$ means 
that at $(x,y) \in \Omega$ and at time $t$ there is open water. Furthermore, let $J = (0,T)$ for $0<T\leq \infty$, let $\Delta$ be the Laplacian, $d_h>0$ and $d_a>0$ be constants and for 
$f \in C^1([0,\infty);\R)$ define the terms  $S_h$ and $S_a$ by 
 \begin{align}
S_h &=  f \big(\frac{h}{a}\big) a + (1-a)f(0) \label{def:sh}, \\
S_a &= \begin{cases*} \frac{f(0)}{\kappa}(1-a),& if $f(0) > 0$ \\ 0, \quad & if $f(0) < 0$ 
    \end{cases*}
    \quad + \quad \begin{cases*} 0,& if $S_h > 0$, \\ \frac{a}{2 h}S_h, & if $S_h < 0$.
    \end{cases*} \label{def:sa}
\end{align}
The system is finally completed by Dirichlet boundary conditions for $u$ 
and Neumann boundary conditions for $h$ and $a$.   

Given $0<T\leq \infty$ and $J= (0,T)$,  the complete set of equations describing sea ice dynamics by Hibler's model then reads as  
\begin{equation}\label{eq:cs}
  \left\{ \begin{array}{rll}
	m(\dot{u} + u \cdot \nabla u) & = \divergence\sigma_\delta - m \ccor n \times u - m g \nabla H + \tatm + \tocean, & \; x\in\Omega, \ t\in J, \\[2mm]
	\dot{h} + \divergence(u h) & =S_h + d_h \Delta h, & \; x\in\Omega, \ t \in J, \\[2mm]
	\dot{a} + \divergence(u a) & =S_a + d_a \Delta a, & \; x\in\Omega, \ t \in J, \\[2mm]
	u=\frac{\partial h}{\partial\nu} & = \frac{\partial a}{\partial\nu} = 0, & \; x\in\partial\Omega, \ t\in J, \\[2mm]
	u(0,x)&=u_0(x), \quad h(0,x)=h_0(x), \quad a(0,x)=a_0(x), & \; x\in\Omega.
	\end{array} \right.
\end{equation}
Note that $m = \rice h$ and $h$ is subject to \eqref{eq:hkappa}. 

To formulate our main well-posedness result for the system \eqref{eq:cs}, we first rewrite it as a quasilinear evolution equation and introduce a setting as follows. Denoting  the 
principle variable
of the system by $v=(u,h,a)$, we rewrite system \eqref{eq:cs} as a quasilinear evolution equation of the form 
\begin{equation}\label{eq:aqle}
v' + A(v)v = F(v), \quad t>0, \quad v(0)=v_0.
\end{equation}
Here $v$ belongs to the ground space $X_0$ defined by 
$$
X_0 = L_q(\Omega;\R^2) \times L_q(\Omega) \times L_q(\Omega),
$$
where $1<q<\infty$. The regularity space will be
$$
X_1 = \{u \in H^{2}_q(\Omega;\R^2): u=0 \mbox{ on } \partial\Omega\} \times   \{h \in H^{2}_q(\Omega): \partial_\nu h =0 \mbox{ on } \partial\Omega\} \times  
\{a \in H^{2}_q(\Omega): \partial_\nu a =0 \mbox{ on } \partial\Omega\}.
$$  
Furthermore, the quasilinear operator $A(v)$ is given by the upper triangular matrix
\begin{align}\label{eq:A(v)}
A(v) &=
    \begin{pmatrix}
    \frac{1}{\rho_{\text{ice}} h} A^H_D (\nabla u,P(h,a)) & \frac{\partial_{h} P(h,a)}{2 \rho_{\text{ice}} h}\nabla & \frac{\partial_{a} P(h,a)}{2 \rho_{\text{ice}} h}\nabla \\
    0 & - d_h \Delta_N & 0\\
    0 & 0 & - d_a \Delta_N
    \end{pmatrix}.
\end{align}
Here $A^H_D$ denotes the realization of Hibler's operator subject to Dirichlet boundary conditions on $L_q(\Omega;\R^2)$, introduced and defined precisely in Section~\ref{sec:3}, and $\Delta_N$ the 
Neumann Laplacian on $L_q(\Omega)$ defined by $\Delta_N = \Delta$ with $D(\Delta_N) = \{h \in H^{2}_q(\Omega): \partial_\nu h =0 \mbox{ on } \partial\Omega\}$. 
The semilinear part $F(v)$ is defined by 
\begin{align}\label{eq:F(v)}
 F(v)
    & = \begin{pmatrix}
   - u \cdot \nabla u - \ccor n \times u -  g \nabla H + \frac{c_1}{h}|\Uatm|\Uatm +  \frac{c_2}{h}|\Uocean-u|(\Uocean -u) \\ 
    - \divergence(u h) + S_h \\ 
     - \divergence(u a) + S_a
    \end{pmatrix},
\end{align}
where $c_1 = \ratm\Catm\Ratm \rice^{-1}$ and  $c_2 = \rocean\Cocean\Rocean\rice^{-1}$ and $\Uatm$ and $\Uocean$ are given functions. 

We consider solutions $v$ within the class
$$
v\in H^1_{p,\mu}(J;X_0)\cap L_{p,\mu}(J;X_1) =: \E_1(J),
$$
where $J=(0,T)$ as above is an interval and $\mu\in (1/p,1]$ indicates a time weight. More precisely,
$$
v\in H^{k}_{p,\mu}(X_l) \quad \Leftrightarrow \quad t^{1-\mu} v\in H^k_p(X_l),\quad k,l=0,1.
$$
The time trace space of this class is given by
\begin{equation}
X_{\gamma,\mu} = (X_0,X_1)_{\mu-1/p,p} = D_{A^H_D(v_0)}(\mu-1/p,p) \times D_{\Delta_N}(\mu-1/p,p)  \times D_{\Delta_N}(\mu-1/p,p)
\end{equation}
provided $p \in (1,\infty)$ and $\mu \in (1/p,1]$. Note that (see e.g., Section 7 of  \cite{AF03}) 
\begin{align}\label{eq:besov}
X_{\gamma,\mu}\hookrightarrow B_{qp}^{2(\mu-1/p)}(\Omega)^{4} \hookrightarrow C^1(\overline{\Omega})^{4}
\end{align}
provided
\begin{equation}\label{eq:pq}
\frac12 + \frac{1}{p} +\frac{1}{q}<\mu\leq 1.
\end{equation}
It is well-known that for $\mu$ satisfying \eqref{eq:pq} the above real interpolation spaces can be characterized  as
\begin{align*}
u \in D_{A^H_D(v_0)}(\mu-1/p,p) &\Leftrightarrow u \in B^{2\mu-2/p}_{qp}(\Omega)^2, u = 0 \mbox{ on } \partial \Omega,\\ 
h \in D_{\Delta_N}(\mu-1/p,p) &\Leftrightarrow h \in B^{2\mu-2/p}_{qp}(\Omega), \partial_\nu h =0 \mbox{ on } \partial \Omega. \\
\end{align*} 
For brevity we set $X_\gamma:=X_{\gamma,1}$. Moreover, let $V_\mu$ be an open subset of $X_{\gamma,\mu}$ such that all 
\begin{equation}\label{def:vmu}
(u,h,a) \in V_\mu  \; \mbox{ satisfy } \;  h \geq \kappa \; \mbox{ for some } \;  \kappa >0 \, \mbox{ and } \, a \in [0,1].     
\end{equation}

\begin{theorem}{\rm (Well-Posedness of Hibler's sea ice model).} \label{thm:local} \\
Let $\Omega \subset \R^2$ be a bounded domain with boundary of class $C^{2}$ and for $\delta>0$ let $\sigma_\delta$ be defined as in \eqref{defsigmadelta}. 
Assume that $1<p,q <\infty$ and that  $\mu \in (1/p,1]$ are subject to \eqref{eq:pq} and  
let $v_0  \in V_\mu$, where $V_\mu$ is as in \eqref{def:vmu}.  

a) Then there exist $\tau=\tau(v_0)>0$ and $r = r(v_0)>0$  
with $\overline{B}_{X_{\gamma,\mu}}(v_0,r) \subset V_\mu$ such that  equation \eqref{eq:aqle}, i.e.,  equations \eqref{eq:cs},  \eqref{def:tau1}, \eqref{def:tau2}, \eqref{def:P}, \eqref{def:sh} and 
\eqref{def:sa},  has a unique solution
$$
v(\cdot,v_1)\in H^1_{p,\mu}(0,\tau;X_0)\cap L_{p,\mu}(0,\tau;X_1) \cap C([0,\tau);V_\mu)
$$
for each initial value $v_1 \in \overline{B}_{X_{\gamma,\mu}}(v_0,r)$. Moreover, there exists $C=C(v_0)$ such that
$$
\Vert v(\cdot,v_1) - v(\cdot,v_2)\Vert_{\E_1(0,\tau)} \leq C \Vert v_1-v_2\Vert_{X_{\gamma,\mu}}
$$
for all $v_1,v_2 \in \overline{B}_{X_{\gamma,\mu}}(v_0,r)$. 
In addition,  
\begin{align*}
t\partial_t v \in  H^1_{p,\mu}(0,\tau;X_0)\cap L_{p,\mu}((0,\tau);X_1), 
\end{align*}
i.e.,\ the solution regularizes instantly in time. In particular,  
$$ 
v \in C^1([b,\tau];X_{\gamma,1})  \cap C^{1-1/p}([b,\tau];X_1) 
$$
for any $b \in (0,\tau)$.

b) The solution $v= v(v_0)$  exists on a maximal time interval $J(v_0) = [0,t^+(v_0))$, which is characterized by the following alternatives:
\begin{itemize}
\item[i)] global existence, i.e., $t_+(v_0) = \infty$,
\item[ii)] $\lim_{t \to t_+(v_0)} \text{dist}_{X_{\gamma,\mu}} \big(v(t),\partial V_\mu\big) =0$,
\item[iii)] $\lim_{t \to t_+(v_0)} v(t)$ does not exist in $X_{\gamma,\mu}$.  
\end{itemize}
\end{theorem}

\begin{remarks}{\rm
a) Assuming $1<p,q<\infty$ and  $\mu \in (1/p,1]$ subject to \eqref{eq:pq},  the smoothness condition required for  the initial data $v_0=(u_0,h_0,a_0)$ in Theorem \ref{thm:local} can be characterized as 
\begin{align*}
u_0 \in B^{2\mu-2/p}_{qp}(\Omega)^2, u_0 = 0 \mbox{ on } \partial \Omega, \;
h_0 \in B^{2\mu-2/p}_{qp}(\Omega), \partial_\nu h_0 =0 \mbox{ on } \partial \Omega, \; 
a_0 \in B^{2\mu-2/p}_{qp}(\Omega), \partial_\nu a_0 =0 \mbox{ on } \partial \Omega. 
\end{align*} 
b) These conditions are  in particular satisfied if $(u_0, h_0, a_0) \in H^{1+2/q+s}_q(\Omega)^4$ for some $s>0$ satisfy the above boundary conditions.  

}
\end{remarks}

Assuming  that $h_*$ and $a_*$ are constant in time and space, we observe that  $(0,h_\ast,a_\ast)$ are trivial equilibria for equation \eqref{eq:aqle} subject to vanishing forcing terms.   
We proceed by showing  that  the equilibrium $(0,h_\ast,a_\ast)$ is stable in $X_{\gamma,\mu}$ and the unique solution of \eqref{eq:aqle} exists globally for initial data close to the 
aforementioned equilibrium provided $\delta$ is chosen small enough and the external forces vanish. 

\begin{theorem}\label{thm:global}
There exists $\delta^* > 0$ such that for all $\delta \in (0,\delta^*)$ and for $h_\ast$ and $a_\ast$ as above, the equilibrium $v_\ast = (0, h_\ast, a_\ast)$ is stable in $X_{\gamma,\mu}$ and there 
exists $r > 0$ such that the unique solution $v$ of \eqref{eq:aqle} without forcing terms and with initial value $v_0 \in X_{\gamma,\mu}$ fulfilling $\Vert v_0 - v_\ast \|_{X_{\gamma,\mu}} < r$ exists 
on $\R_+$ and converges at an exponential rate in $X_{\gamma,\mu}$ to some equilibrium $v_\infty$ of \eqref{eq:aqle} as $t \to \infty$.
\end{theorem}

\begin{remark}{\rm
Equations \eqref{eq:cs} show that without forcing, the solution tends to the equilibria  $h_\ast = \frac{1}{|\Omega|}\int_{\Omega} h(0) dx$ and $a_\ast= \frac{1}{|\Omega|}\int_\Omega a(0) dx$ determined by the initial mean values of $h$ and $a$, respectively.
}
\end{remark}

\vspace{.2cm}
\section{Hibler's ice stress viewed as a second order quasilinear operator}\label{sec:3}

In this section we interpret the term $\divergence \sigma$ as a quasilinear second order operator. To this end, denote by  $\eps =(\eps)_{ij}$ the deformation or rate of strain tensor
and define the map $\mathbb{S} \colon \R^{2\times 2} \to \R^{2\times 2}$ in such a way  that 
\begin{align*}
    \mathbb{S} \eps
    &= \begin{pmatrix}
    (1 + \frac{1}{e^2}) \eps_{11} + (1 - \frac{1}{e^2}) \eps_{22} & \frac{1}{e^2} (\eps_{12} + \eps_{21}) \\
    \frac{1}{e^2} (\eps_{12} + \eps_{21}) & (1 - \frac{1}{e^2}) \eps_{11} + (1 + \frac{1}{e^2}) \eps_{22}
    \end{pmatrix}.
	\end{align*}
	If $\eps \in \R^{2\times 2}$ is identified with the vector $(\eps_{11}, \eps_{12}, \eps_{21}, \eps_{22})^T \in \R^4$, $\mathbb{S}$ corresponds to the symmetric positive semi-definite matrix
	\begin{align*}
	    \mathbb{S} = \bigl(\mathbb{S}^{kl}_{ij}\bigr) 
	    &= \begin{pmatrix}
	    1+\frac{1}{e^2} & 0 & 0 & 1-\frac{1}{e^2} \\ 
		0 & \frac{1}{e^2}  & \frac{1}{e^2} & 0   \\
		0 & \frac{1}{e^2} & \frac{1}{e^2} & 0  \\
	     1-\frac{1}{e^2} & 0 & 0 & 1 + \frac{1}{e^2} 
	    \end{pmatrix},
	\end{align*}	
	and we obtain 
	\begin{align}\label{eq:triangS}
	\triangle^2 (\eps) & = \eps^T \mathbb{S} \eps = \sum \limits_{i,j,k,l=1}^2 \eps_{ik} \mathbb{S}_{ij}^{kl} \eps_{jl}
	 = (\eps_{11} + \eps_{22})^2 + \frac{1}{e^2} (\eps_{11} - \eps_{22})^2 + \frac{1}{e^2} (\eps_{12} + \eps_{21})^2. 
	\end{align}
The stress tensor $\sigma=\sigma(\eps,P)$ can then be represented as 
\begin{align}\label{eqintroofstresstensors}
   \sigma(\eps,P) = S(\eps, P) - \frac{P}{2}I, \quad \text{where} \quad S(\eps,P) &:= \frac{P}{2}\frac{\mathbb{S} \eps}{\triangle(\eps)}.
\end{align}
As explained in Section 2, for $\delta>0$ we then substitute $S$ by
\begin{align}\label{eqregularizedstresstensor}
    S_\delta = S_\delta (\eps,P) &:= \frac{P}{2}\frac{\mathbb{S} \eps}{\triangle_\delta (\eps)},
\end{align} 
and define Hibler's operator as 
$$ \cA^H u := - \divergence S_\delta (u) = - \divergence\Bigl(\frac{P}{2}\frac{\mathbb{S} \eps}{\sqrt{\delta + \eps^T \mathbb{S} \eps}}\Bigr). 
$$
Employing product and chain rule as well as symmetries of $\mathbb{S}$, we infer that
\begin{align*}
    \divergence\Bigl(\frac{P}{2}\frac{\mathbb{S} \eps}{\sqrt{\delta + \eps^T \mathbb{S} \eps}}\Bigr)_i
 &= \sum \limits_{j,k,l=1}^2 \frac{P}{2}\frac{1}{\triangle_\delta (\eps)}\Bigl(\mathbb{S}_{ij}^{kl} - \frac{1}{\triangle_\delta ^2 (\eps)}(\mathbb{S} \eps)_{ik} (\mathbb{S} \eps)_{lj} \Bigr) \partial_k \eps_{jl}
    + \frac{1}{2 \triangle_\delta (\eps)}\sum \limits_{j=1}^2 (\partial_j P) (\mathbb{S} \eps)_{ij}
\end{align*}
for $i=1,2$. Exploiting symmetries of $\mathbb{S}$ and $\eps$ once again,
we conclude that
\begin{align}\label{eq:hiblerop}
 (\cA^H u)_i
 &= \sum \limits_{j,k,l=1}^2 \frac{P}{2}\frac{1}{\triangle_\delta (\eps)}\Bigl(\mathbb{S}_{ij}^{kl} - \frac{1}{\triangle_\delta ^2 (\eps)}(\mathbb{S} \eps)_{ik} (\mathbb{S} \eps)_{jl} \Bigr) D_k D_l u_j - 
\frac{1}{2 \triangle_\delta (\eps)}\sum \limits_{j=1}^2 (\partial_j P) (\mathbb{S} \eps)_{ij}
\end{align}
for $i=1,2$ and $D_m = - \mathrm{i} \partial_m$.

We denote the coefficients of the principal part of $\cA^H$ by 
\begin{align}\label{eq:aijkl}
{a_{ij}^{kl}(\nabla u,P) := \frac{P}{2}\frac{1}{\triangle_\delta (\eps)}\bigl(\mathbb{S}_{ij}^{kl} - \frac{1}{\triangle_\delta ^2 (\eps)}(\mathbb{S} \eps)_{ik} (\mathbb{S} \eps)_{jl} \bigr)}.
\end{align}

In view of the symmetries of $\mathbb{S}$ and $\mathbb{S} \eps$ we conclude that
\begin{align}\label{eq:symofaijkl}
{a_{ij}^{kl} = a_{ji}^{lk} = a_{kl}^{ij} = a_{kj}^{il} = a_{il}^{kj} = a_{lk}^{ji}}.
\end{align}

For given $v_0=(u_0, h_0, a_0) \in V_\mu$  with $\mu > \frac12 +\frac{1}{p} + \frac{1}{q}$, let
\begin{equation}\label{def:ahv0}
[\cA^H(v_0)u]_i = \sum_{j,k,l=1}^2 a_{ij}^{kl}(\nabla u_0,P(h_0,a_0))D_kD_lu_j  - \frac{1}{2 \triangle_\delta (\eps(u_0))}\sum \limits_{j=1}^2 (\partial_j P(h_0,a_0)) (\mathbb{S} \eps(u))_{ij}
\end{equation}
be Hibler's operator with frozen coefficients. 
The representation in \eqref{eq:aijkl} shows that the principal coefficients $a^{kl}_{ij}(\nabla u_0,P(h_0,a_0))$ 
of $\cA^H(v_0)$, as well as lower-order terms, depend smoothly on $u_0$, $h_0$ and $a_0$ with respect to the $C^1$-norm.  Moreover, the embedding \eqref{eq:besov}  
yields that 
they lie in $C(\overline{\Omega})$. 
  
\vspace{.2cm}
\section{Hibler's operator: Ellipticity and Maximal Regularity}\label{sec:4}

In this section we show that Hibler's operator $\cA^H(v_0)$ given as in \eqref{def:ahv0} defines a strongly elliptic operator and, when subject to Dirichlet boundary conditions, satisfies the 
Lopatinskii-Shapiro condition. This implies then that the  $L_q$-realization $A^H_D(v_0)$ of $\cA^H(v_0)$ given by 
\begin{equation}\label{eq:lqreal}
[A^H_D(v_0)]u:= [\cA^H(v_0)]u \mbox{ for } u \in D(A^H_D(v_0)):= \{u \in H^{2}_q(\Omega;\R^2) : u = 0 \ \text{on} \, \partial \Omega\}
\end{equation}
satisfies the maximal $L_q$-regularity property and furthermore that $A^H_D(v_0)$ admits a bounded $H^\infty$-calculus on  $L_q(\Omega;\R^2)$.        

For $\theta \in (0,\pi]$ let  ${\Sigma_\theta := \{z \in \C \setminus \{0\} : \vert \arg z \vert < \theta\}}$ be a sector in the complex plane and let $D=-i(\partial_1,\ldots,\partial_n)$.  
We start by recalling from \cite{DHP03} that for $x \in \R^n$ an  operator $\cB$ of the form $\cB(x,D)=\sum_{|\alpha|\le 2} b_\alpha(x)D^\alpha$ with continuous top order coefficients 
$b_\alpha \in \cL(E)$, $E$ an arbitrary Banach space, is said to be {\em parameter-elliptic of angle $\phi \in (0,\pi]$} if the spectrum ${\sigma(\cB_{\#}(x,\xi))}$ of the symbol of the principal part 
$\cB_{\#}(x,\xi) = \sum_{|\alpha|=2} b_\alpha(x)\xi^{\alpha}$ satisfies 
\begin{align*}
    \sigma(\mathcal{B}_{\#}(x,\xi)) \subset \Sigma_{\phi}
\end{align*}
for every $x \in \R^n$ and every $\xi \in \R^n$ with $\vert \xi \vert = 1$. We call ${\phi_{\mathcal{B}} = \inf\{\phi : \sigma(\mathcal{B}_{\#}(x,\xi)) \subset \Sigma_\phi\}}$ the {\em angle of 
ellipticity}  of 
$\cB$. Moreover, the operator $\cB(x,D)$ is called {\em normally elliptic} if it is parameter-elliptic of angle $\phi_{\cB} < \pi/2$. If $E$  is a Hilbert space, an operator $\cB$ of the above 
form $\cB(x,D)=\sum_{|\alpha|= 2} b_\alpha(x)D^\alpha$ is called {\em strongly elliptic} if there exists a constant $c > 0$ such that
\begin{align}\label{def:se}
    \Ret\; (\cB_{\#}(x,\xi) w \vert w)_E  \ge c \Vert w \Vert_E^2
\end{align}
for all $x \in \R^n$, $\xi \in \R^n$ with $\vert \xi \vert = 1$ and all  $w \in E$.  Here  $(\cdot \vert\cdot)_E$ denotes the inner product on $E$. To understand this condition, let $n(L)$ be the numerical range of a bounded
linear operator on $E$, i.e., $n(L)$ is the closure of the set consisting of all $z \in \C$ such that $z=(Lw \vert w)_E$ for some $w \in E$ with $\Vert w \Vert_E=1$. Since $\sigma(L) \subset n(L)$ we see that 
every strongly elliptic operator  $\cB$ is parameter-elliptic of angle $\phi_\cB < \pi/2$, hence even normally elliptic. 

Consider now the special case of homogeneous differential operators acting on $\C^n$-valued functions as 
$$ 
[\cB(x,D)v(x)]_i := \sum_{j,k,l=1}^n b_{ij}^{kl}(x) D_k D_l \, v_j(x), \quad x \in \Omega.
$$
Here $\Omega \subset \R^n$ denotes a domain with boundary of class $C^2$. Its symbol is defined as 
$$ 
\big(\cB_{\#}(x,\xi)\big)_{ij} := \sum_{k,l=1}^n b_{ij}^{kl}(x) \xi_k \xi_l, \quad x \in \Omega.
$$



We now show that $\cA^H(v_0)$ is strongly elliptic provided $v_0 \in V_\mu$.  

\begin{proposition}\label{prop:strongell}
Let $p,q \in (1,\infty)$ and  $\mu \in (\frac{1}{p},1]$ such that \eqref{eq:pq} holds. Then, for fixed $v_0 \in V_\mu$, the principal part of Hibler's operator $\cA^H(v_0)$ defined as in \eqref{def:ahv0} 
is strongly elliptic and moreover parameter-elliptic of angle $\phi_{\cA^H(v_0)} =0$. 
\end{proposition}

\begin{proof}
Recall that the principal part of $\cA^H(v_0)$ is given by 
\begin{align*}
    \mathcal{A}_{\#}^H(x,\xi) := \sum_{k,l=1}^2 a_{ij}^{kl}(x) \xi_k \xi_l, \quad x \in \Omega,
\end{align*}
with
$$
{a_{ij}^{kl}(\nabla u_0,P_0) := \frac{P_0}{2}\frac{1}{\triangle_\delta (\eps)}\bigl(\mathbb{S}_{ij}^{kl} - \frac{1}{\triangle_\delta ^2 (\eps)}(\mathbb{S} \eps)_{ik} (\mathbb{S} \eps)_{jl} \bigr)}
$$
as in \eqref{eq:aijkl} and $P_0=P(h_0,a_0)$. 
Taking into account the underlying symmetries we see that the symbol of the principal part of $\cA^H(v_0)$ is given by 
\begin{align}\label{eq:matrix}
    \mathcal{A}_{\#}^H(x,\xi)
    &= \begin{pmatrix}
    a_{11}^{11} \xi_1 ^2 + 2 a_{11}^{12} \xi_1 \xi_2 + a_{11}^{22} \xi_2 ^2 & a_{11}^{12} \xi_1 ^2 + (a_{12}^{12} + a_{11}^{22}) \xi_1 \xi_2 + a_{12}^{22} \xi_2 ^2 \\
    a_{11}^{12} \xi_1 ^2 + (a_{12}^{12} + a_{11}^{22}) \xi_1 \xi_2 + a_{12}^{22} \xi_2 ^2 & a_{11}^{22} \xi_1 ^2 + 2 a_{12}^{22} \xi_1 \xi_2 + a_{22}^{22} \xi_2 ^2
    \end{pmatrix}.  
\end{align}
For given $d \in \R^{2\times 2}$, we use the notation 
\begin{align*} d_I := d_{11} + d_{22}, \quad d_{II} := d_{11} - d_{22}, \quad d_{III} := \frac{d_{12}+ d_{21}}{2}.
	\end{align*}
To verify condition \eqref{def:se}, first recall that for any given $d \in \R^{2\times 2}$, (see \eqref{eq:triangS}), 
\begin{align}\label{eq:ellest1}
	d^T \mathbb{S} d & = d_I^2 + \frac{1}{e^2}(d_{II}^2 + 4d^2_{III}) =: \triangle^2(d).
	\end{align} 	
Furthermore, using Young's inequality, we estimate
\begin{align}\label{eq:ellest2}	
\bigl( d^T (\mathbb{S} \eps) \bigr)^2 
& = \Bigl( d_{I}\eps_{I} + \frac{d_{II}\eps_{II}}{e^2} + \frac{4d_{III}\eps_{III}}{e^2} \Bigr)^2 \\ \nonumber
& \leq d_I^2 \Bigl(\eps_I^2 + \frac{\eps_{II}^2}{e^2} + \frac{4\eps_{III}^2}{e^2} \Bigr) + \frac{1}{e^2}  \bigl( d_{II}^2 + 4 d_{III}^2 \bigr) \Bigl(\eps_I^2 + \frac{\eps_{II}^2}{e^2} + \frac{4\eps_{III}^2}{e^2}\Bigr) \\ \nonumber
& =  \triangle^2(d) \triangle^2(\eps). 
 \end{align}
	 Due to our assumptions on $v_0$, the function $\frac{P_0}{2\triangle_\delta(\eps)^3}$
is real-valued, bounded, continuous and positively bounded from below by a constant $c_{\delta,\kappa,\alpha} > 0$. Thus,  combining \eqref{eq:ellest1} and \eqref{eq:ellest2}, for all $d \in \R^{2\times2}$, we obtain
 \begin{align}\label{eq:ellest3}
	\sum_{i,j,k,l=1}^2  a_{ij}^{kl} d_{ik} d_{jl} & = \frac{P_0}{2\triangle_\delta(\eps)^3} \bigl( \triangle_\delta(\eps)^2 d^T \mathbb{S} d - (d^T \mathbb{S} \eps)^2 \bigr)
	\geq c_{\delta,\kappa,\alpha} \delta \triangle^2(d).  
	 \end{align}
	 We can now verify condition \eqref{def:se}. Given $\xi \in \R^2$ and $\eta \in \C^2$ with $|\xi| = |\eta| =1$, set $\eta_i =: x_i + \mathrm{i} y_i$ for $i = 1,2$. 
	 Because of the symmetries of $(a^{ij}_{kl})$ as pointed out in \eqref{eq:symofaijkl} and using \eqref{eq:ellest3}, we derive 
	 \begin{align*}
		\Ret(\mathcal{A}_{\#}^{H}(x,\xi)\eta \vert \eta)  & = \Ret \sum_{i,j,k,l=1}^2  a_{ij}^{kl} (\xi \otimes \eta)_{jl} \overline{(\xi \otimes \eta)}_{ik} \\
		& = \sum_{i,j,k,l=1}^2  a_{ij}^{kl} (\xi \otimes x)_{jl} (\xi \otimes x)_{ik} +  a_{ij}^{kl} (\xi \otimes y)_{jl} (\xi \otimes y)_{ik} \\
		& \geq c_{\delta,\kappa,\alpha} \delta \bigl((\triangle^2(\xi \otimes x) + \triangle^2(\xi \otimes y) \bigr). 
		 \end{align*} 
		 Moreover, using $|\xi| = 1$,
		 \begin{align*}
			 \triangle^2(\xi \otimes x) = (\xi \cdot x)^2 + \frac{1}{e^2} \Vert x \Vert^2,
			 \end{align*}
			 so using $|\eta| = 1$, 
			 $$  \triangle^2(\xi \otimes x) +  \triangle^2(\xi \otimes y) \geq \frac{1}{e^2}, $$
			 and thus $\cA^H(v_0)$ is strongly elliptic with an ellipticity constant $\geq \frac{c_{\delta,\kappa,\alpha} \delta }{e^2} $.

To prove  parameter-ellipticity of $\cA^H(v_0)$, we first note that due to \eqref{eq:matrix}, strong ellipticity implies normal ellipticity, and by symmetry of $\mathcal{A}_{\#}^H$, we conclude that 
${\sigma(\mathcal{A}_{\#}^H)(x,\xi) \subset \R_{+}}$ is valid for every ${x \in \overline{\Omega}}$ and ${\xi \in \R^2}$ with ${\vert \xi \vert = 1}$. This implies parameter-ellipticity of 
$\cA^H(v_0)$ with ${\phi_{\mathcal{A}^H(v_0)} = 0}$.
\end{proof}

The assertion of the following lemma will be crucial  in the proof of the fact that the linearized Hibler operator $\mathcal{A}^H(v_0)$ subject to  Dirichlet boundary conditions satisfies the 
Lopatinskii-Shapiro condition.

\vspace{.2cm}
\begin{lemma}\label{lemma:sne}
Let $p,q \in (1,\infty)$ and  $\mu \in (\frac{1}{p},1]$ such that \eqref{eq:pq} holds. For fixed $v_0 \in V_\mu$, let   $a_{ij}^{kl}$  be the coefficients of the principal part of  Hibler's operator  
$\cA^H(v_0)$ defined as in \eqref{eq:aijkl}.  
Assume that  ${x \in \partial \Omega}$, $\xi$, ${\nu \in \R^2}$ with ${\vert \xi \vert = \vert \nu \vert = 1}$ and ${(\xi \vert \nu) = 0}$ as well as $u$, ${v \in \C^2}$. Then 
\begin{align}
   \Ret\; \Bigl(\sum_{i,j,k,l=1}^2 &  a_{ij}^{kl}(\xi_l u_j - \nu_l v_j)\overline{(\xi_k u_i - \nu_k v_i)}\Bigr) \ge 0 \quad  \mbox{ and }  \label{sne1} \\
\Ret\; \Bigl(\sum \limits_{i,j,k,l=1}^2 & a_{ij}^{kl}(\xi_l u_j - \nu_l v_j)\overline{(\xi_k u_i - \nu_k v_i)}\Bigr) > 0 \quad \mbox{ provided } \quad  \Imt(u \vert v) \neq 0. \label{sne2}
\end{align}
\end{lemma}

\begin{proof}
Let ${x \in \partial \Omega}$, $\xi$, ${\nu \in \R^2}$ with ${\vert \xi \vert = \vert \nu \vert = 1}$ and ${(\xi \vert \nu) = 0}$ as well as $u$, ${v \in \C^2}$.
We introduce the notation ${u_i = x_i + \mathrm{i} y_i}$ and ${v_i = \tilde{x}_i + \mathrm{i} \tilde{y}_i}$, ${i=1,2}$. 
Using the symmetries of $(a^{kl}_{ij})$ as in \eqref{eq:symofaijkl} and the estimate \eqref{eq:ellest3}, we obtain
\begin{align}\label{eq:startproofsne}
    \quad \Ret &\; \Bigl(\sum \limits_{i,j,k,l=1}^2 a_{ij}^{kl}(\xi_l u_j - \nu_l v_j)\overline{(\xi_k u_i - \nu_k v_i)}\Bigr)\\
  & = \sum \limits_{i,j,k,l=1}^2 a_{ij}^{kl}(\xi_l x_j - \nu_l \tilde{x}_j)(\xi_k x_i - \nu_k \tilde{x}_i) + \sum \limits_{i,j,k,l=1}^2 a_{ij}^{kl}(\xi_l y_j - \nu_l \tilde{y}_j)(\xi_k y_i - \nu_k \tilde{y}_i) \notag\\
  & \geq c_{\delta, \kappa,\alpha} \delta \Bigl((\triangle^2(\xi \otimes x - \nu \otimes \tilde{x}) + \triangle^2(\xi \otimes y - \nu \otimes \tilde{y}) \Bigr) \geq 0 \notag.
\end{align}
Thus, condition \eqref{sne1} is satisfied. To verify condition \eqref{sne2}, it remains to consider the case $ = 0$ in the last line and deduce 
\begin{align}\label{eq:Imuv}
	\Imt(u \vert v) = \tilde{x}_1 y_1 - x_1 \tilde{y}_1 +  \tilde{x}_2 y_2 - x_2 \tilde{y}_2 = 0. 
	\end{align}
For general $d \in \R^{2\times 2}$, $\triangle^2(d) = 0$ implies $d_{11} = d_{22} = 0$, so from $ = 0$ in \eqref{eq:startproofsne} we obtain 
\begin{align}\label{eq:triang0}
 \xi_1 x_1 - \nu_1 \tilde{x}_1 =  \xi_2 x_2 - \nu_2 \tilde{x}_2=  \xi_1 y_1 - \nu_1 \tilde{y}_1=  \xi_2 y_2 - \nu_2 \tilde{y}_2 = 0.
 \end{align}
Due to $|\xi| = 1$, either $\xi_1 \neq 0$ or $\xi_2 \neq 0$. Assume $\xi_1 \neq 0$. In view of  $(\xi \vert \nu) = 0$ and $\vert \nu \vert = 1$, this implies $\nu_2 \neq 0$. Thus, from \eqref{eq:triang0}, we obtain 
$$
x_1 = \frac{\nu_1}{\xi_1} \tilde{x}_1, \quad \tilde{x}_2 = \frac{\xi_2}{\nu_2} x_2, \quad y_1 = \frac{\nu_1}{\xi_1} \tilde{y}_1, \quad \tilde{y}_2 = \frac{\xi_2}{\nu_2} y_2.
$$
Plugging this into \eqref{eq:Imuv} immediately yields the claim. The case $\xi_2 \neq 0$ follows analogously. 
%
\end{proof}

\vspace{.2cm}
We proceed by showing that Hibler's operator subject to Dirichlet boundary conditions fulfills the Lopatinskii-Shapiro condition. For the formulation 
of  the latter condition in the context of parabolic boundary value problems subject to general boundary conditions, see e.g. \cite{DHP03, DDHPV04, DHP07}. 

In our  context of Hibler's operator subject to Dirichlet boundary conditions, the    Lopatinskii-Shapiro condition reads as follows: For all  $x_0 \in \partial\Omega$, all $\xi \in \R^2$ with 
$(\xi,\nu(x))=0$, and all $\lambda \in \C$ satisfying $\Ret \lambda \geq 0$ and $|\xi| + |\lambda| \ne 0$, any solution $w \in C_0(\R_+;\C^2)$ of  the ordinary differential equation in $\R_+$ 
\begin{equation}\label{eqodelopatinskiishapiro}
  \left\{ \begin{array}{rll}
	(\lambda + \mathcal{A}_{\#}^H(x_0,\xi - \nu(x_0)D_y))w(y) &= 0, & \; y > 0, \\[2mm]
	 w(0)&=0,
\end{array} \right.
\end{equation}
equals zero. 

\begin{proposition}\label{prop:ls}
Let $p,q \in (1,\infty)$ and  $\mu \in (\frac{1}{p},1]$ such that \eqref{eq:pq} holds. Then, for  fixed $v_0 \in V_\mu$ the principal part of Hibler's operator $\cA^H(v_0)$ 
subject to homogeneous Dirichlet boundary conditions satisfies the Lopatinskii-Shapiro condition.
\end{proposition}

\begin{proof}
Taking the inner product of the above equation with a solution $w$ we obtain 
\begin{align*}
    0 = \lambda (w(y) \vert w(y)) + (\mathcal{A}^H_{\#}(x_0,\xi - \nu(x)D_y)w(y) \vert w(y)).
\end{align*} 
Integrating over $\R_{+}$ and integrating by parts yields
\begin{align}\label{eq:w} 
   0  &= \lambda \Vert w \|_2 ^2 + \int_{0}^{\infty} \sum_{i,j,k,l=1}^2  a_{ij}^{kl}(\xi_l - \nu_l(x)D_y)w_j(y)\overline{(\xi_k - \nu_k(x)D_y)w_i(y)} \,d y.
\end{align}
Our aim is to deduce from \eqref{eq:w} that $w \equiv 0$ for each solution $w \in H_2^2 (\R_{+};\C^2)$ and thus for $w \in C_b ^1(\R_{+};\C^2)$. 

Taking real parts in \eqref{eq:w}, we see  by \eqref{sne1} that 
\begin{align}\label{eq:int=0}
    \int_{0}^{\infty} \Ret \sum \limits_{i,j,k,l=1}^2  a_{ij}^{kl}(\xi_l w_j(y) - \nu_l(x)D_y w_j(y))\overline{(\xi_k w_i(y) - \nu_k(x)D_y w_i(y))} \,d y = 0.
\end{align}
Assuming $\frac{d}{d y} \vert w(y) \vert ^2 = 0$ for all $y > 0$ yields that $\vert w(y) \vert$ is constant on $\R_{+}$ and that consequently  $w(y)=0$ on $\R_{+}$. 

Calculating $\frac{d}{d y} \vert w(y) \vert ^2$ yields  ${\frac{d}{d y} \vert w(y) \vert ^2 = 2 \Ret \bigl(\frac{d}{d y}w(y) \vert w(y)\bigr) = - 2 \Imt (D_y w(y) \vert w(y))}$. 
Suppose now that there exists ${y_0 > 0}$ such that $\frac{d}{d y} \vert w(y) \vert ^2$ does not vanish at $y_0$. Then, by smoothness of $w$, 
there exists a neighborhood ${U \subset \R_{+}}$ of $y_0$ with ${\frac{d}{d y} \vert w(y) \vert ^2 \neq 0}$ for all ${y \in U}$. 
Then also ${\Imt (D_y w(y) \vert w(y)) \neq 0}$ for all ${y \in U}$. 
Setting ${u:= D_y w(y) \in \C^2}$ and ${v:= w(y) \in \C^2}$ for  ${y \in U}$ we see that  ${\Imt (u,v) \neq 0}$ for all ${y \in U}$ as well. Consequently, \eqref{sne2}  and \eqref{sne1} yield  
\begin{align*}
    \int_{0}^{\infty} \Ret \sum \limits_{i,j,k,l=1}^2   a_{ij}^{kl}(\xi_l w_j(y) - \nu_l(x)D_y w_j(y))\overline{(\xi_k w_i(y) - \nu_k(x)D_y w_i(y))} \,d y >  0. 
\end{align*}
Combining this with relation \eqref{sne1} contradicts, however, condition \eqref{eq:int=0}. Thus $w \equiv 0$. 
\end{proof}

\vspace{.2cm}
We recall that for $1<r<\infty$, the $L_r$-realization $A^H_D(v_0)$ of $\cA^H(v_0)$ subject to Dirichlet boundary conditions is given by 
\begin{equation}\label{def:lqreal}
[A^H_D(v_0)]u = [\cA^H(v_0)]u \mbox{ for } u \in D(A^H_D(v_0)):= \{u \in H^{2}_r(\Omega;\R^2) : u = 0 \ \text{on} \, \partial \Omega\}.
\end{equation}
We will now prove the maximal $L_s$-regularity property for $A^H_D(v_0)$ in the $L_r$-setting, where $1 < s,r < \infty$. 

\vspace{.2cm}
\begin{theorem}\label{thm:mr}
Let $p,q,r,s \in (1,\infty)$ and  $\mu \in (\frac{1}{p},1]$ such that \eqref{eq:pq} holds and let $v_0 \in V_\mu$ be fixed.  Then there exists $\omega_0 \in \R$ such that for all 
$\omega>\omega_0$ \\
a) $A^H_D(v_0) + \omega$ has the property of maximal $L_s[0,\infty)$-regularity on $L_r(\Omega;\R^2)$, \\
b) $A_D^H(v_0)+ \omega$ admits  a bounded $H^\infty$-calculus on $L_r(\Omega;\R^2)$. 

\end{theorem}

\begin{proof}
By Proposition \ref{prop:strongell}, for  fixed $v_0 \in V_\mu$, the principal part of Hibler's operator $\cA^H(v_0)$ is a  parameter-elliptic operator with continous and bounded coefficients on 
$\overline{\Omega}$ having angle of ellipticity $\phi_{\cA^H(v_0)}=0$.  Furthermore, Proposition \ref{prop:ls} tells us that the principal part of Hibler's operator $\cA^H(v_0)$ subject to 
homogeneous Dirichlet boundary conditions satisfies the Lopatinskii-Shapiro condition. Since the coefficients of the lower order terms 
of $\cA^H(v_0)$ are smooth, the first assertion follows from the results in \cite{DHP03, DHP07}.   

          
The second assertion follows by the results in \cite{DDHPV04} provided the top-order coefficients of $\cA^H_D(v_0)$ are H\"older continuous. The latter condition is satisfied due to the embedding 
$B_{qp}^{2\mu-2/p}(\Omega) \hookrightarrow C^{1,\alpha}(\overline{\Omega})$.    
\end{proof}

\vspace{.2cm}
\section{Functional analytic properties of Hibler's operator}

Hibler's operator enjoys many interesting properties, some of  which we collect in the following.

\begin{proposition}\label{prop:fa}
Let ${\Omega \subset \R^2}$ be a bounded domain with $C^2$-boundary, ${1 < p,q,r < \infty}$, ${\mu \in (\frac{1}{p},1]}$  such that \eqref{eq:pq} is satisfied and for 
$v_0 \in V_\mu$ let Hibler's operator $A^H_D(v_0)$ on $L_r(\Omega;\R^2)$ with domain $D(A_D^H(v_0))$ be defined as in \eqref{def:lqreal}. Then
\begin{enumerate}
\item[a)]  $ -A_D^H(v_0)$ generates an analytic semigroup $e^{-tA^H_D(v_0)}$ on $L_r(\Omega;\R^2)$,
\item[b)]  $A_D^H(v_0)$ is an operator with compact resolvent, 
\item[c)] the spectrum $\sigma(A^H_D(v_0))$ of $A^H_D(v_0)$ viewed as an operator on $L_r(\Omega)^2$ is $r$-independent,
\item[d)] for $\alpha \in (0,1)$ and $\omega$ as in Theorem \ref{thm:mr}
$$
D((A^H_D(v_0)+\omega)^\alpha) \simeq [L_r(\Omega;\R^2),D(A_D^H(v_0))]_{\alpha} = 
\begin{cases} 
\{u \in H^{2\alpha}_r(\Omega): u|_{\partial\Omega} =0\}, \alpha  \in (1/2r,1], \\
H^{2\alpha}_r(\Omega): $ $\alpha  \in [0,1/2r),
\end{cases}
$$
\item[d)] The Riesz transform $\nabla (A^H_D(v_0)+\omega)^{-1/2}$ of $A^H_D(v_0)$ is bounded on $L_r(\Omega)^2$.
\end{enumerate}
\end{proposition}

\begin{proof}
Assertions a) follows by standard arguments. The compact embedding $D(A^H_D(v_0)) \hookrightarrow L_r(\Omega)^2$ implies that $A^H_D(v_0)$ has compact resolvent and that 
thus $\sigma(A^H_D(v_0))$ is independent of $r \in (1,\infty)$. Assertion d) follows from the fact that $A^H_D(v_0)+\omega$  admits a bounded $H^\infty$-calculus on $L_r(\Omega)^2$ by Theorem~\ref{thm:mr}. 
Finally, assertion e) is obtained by noting that $D((A^H_D(v_0)+\omega)^{1/2}) \subset H^1_r(\Omega)$.     
\end{proof}

\vspace{.2cm}
\section{Proof of Theorem~\ref{thm:local}}\label{seclocexofsol}
We recall from Section 2 that the ground space $X_0$ for $q \in (1,\infty)$ is given by   
\begin{align*}
X_0 = L_q(\Omega;\R^2) \times L_q(\Omega) \times L_q(\Omega) =: X_0^u \times X_0^h \times X_0^a.
\end{align*}
The regularity space $X_1$ is defined as 
$$
X_1 = \{u \in H^{2}_q(\Omega;\R^2): u=0 \mbox{ on } \partial\Omega\} \times   \{h \in H^{2}_q(\Omega): \partial_\nu h =0 \mbox{ on } \partial\Omega\} \times  
\{a \in H^{2}_q(\Omega): \partial_\nu a =0 \mbox{ on } \partial\Omega\}.
$$  
Since we are considering solutions within the class 
$$
v\in H^1_{p,\mu}(J;X_0)\cap L_{p}(J;X_1),
$$
where $J=(0,T)$ with $0<T\leq\infty$ is an interval and $\mu\in (1/p,1]$ indicates a time weight, 
the time trace space of this class is given by
\begin{equation}\label{eq:traceproof}
X_{\gamma,\mu} = (X_0,X_1)_{\mu-1/p,p} = D_{A^H_D(v_0)}(\mu-1/p,p) \times D_{\Delta_N}(\mu-1/p,p)  \times D_{\Delta_N}(\mu-1/p,p) =: X_{\gamma,\mu}^u \times  X_{\gamma,\mu}^h \times  X_{\gamma,\mu}^a 
\end{equation}
provided $p \in (1,\infty)$ and $\mu \in (1/p,1]$. Note that
$$
X_{\gamma,\mu}\hookrightarrow B_{qp}^{2(\mu-1/p)}(\Omega)^{4} \hookrightarrow C^1(\overline{\Omega})^{4}
$$
provided \eqref{eq:pq} is satisfied. 


For $\omega$ as in Theorem \ref{thm:mr} we  now consider the operator $A_\omega(v_0)$ on $X_0$ with domain $X_1$ given by the upper triangular matrix
\begin{align}\label{eq:A(v_0)omega}
A_\omega(v_0) &=
    \begin{pmatrix}
    \frac{1}{\rice h_0}( A^H_D(v_0) + \omega) & \frac{\partial_{h} P(h_0,a_0)}{2 \rice h_0}\nabla & \frac{\partial_{a} P(h_0,a_0)}{2 \rice h_0}\nabla \\
    0 & - d_h \Delta_N & 0\\
    0 & 0 & - d_a \Delta_N
    \end{pmatrix},
\end{align}
as well as 
\begin{equation}\label{eq:F(v)omega} 
F_\omega(v) = F(v) + \frac{1}{\rice h}(\omega,0,0)^T,
\end{equation}
where $F$ is given as in \eqref{eq:F(v)}

\vspace{.2cm}
\begin{lemma}\label{lem:maxreg}
Let  $p,q \in (1,\infty)$,  $\mu \in (1/p,1]$ such that  \eqref{eq:pq} is satisfied and assume that $v_0=(u_0,h_0,a_0) \in V_\mu$.  Let $J = [0,T)$ for some $0< T < \infty$. 
Then $A_\omega(v_0)$ has maximal $L_p(J)$-regularity on $X_0$.   
\end{lemma}

\begin{proof}
By assumption, we have $h_0 \geq \kappa$ for some $\kappa >0$ and \eqref{eq:pq} implies that $1/h_0 \in C^1(\overline{\Omega})$.   
Since $\Delta_N$ as well as $A^H_D(v_0)+\omega$ and $\frac{1}{h_0}(A^H_D(v_0)+\omega)$ have the maximal $L_p(J)$-regularity property on $X_0$  by Theorem~\ref{thm:mr}, the upper triangular 
structure of $A_\omega(v_0)$ implies
that also $A_\omega(v_0)$ has the maximal $L_p(J)$-regularity property on $X_0$.     
\end{proof}

We now show that $(A_\omega,F_\omega) \in C^{1-}(V_\mu;\cL(X_1,X_0) \times X_0)$ for $\mu \in (1/p,1]$  satisfying \eqref{eq:pq} and  where $A$ and $F$ are defined as   in \eqref{eq:A(v_0)omega} and 
\eqref{eq:F(v)omega}. Recall that  $V_\mu$ is an open subset of $X_{\gamma,\mu}$ such that all $(u,h,a) \in V_\mu$ satisfy  $h \geq \kappa$ for some $\kappa >0$.

\vspace{.2cm}   
\begin{lemma}\label{lemma:lipschitz}
Let  $p,q \in (1,\infty)$,  $\mu \in (1/p,1]$ such that  \eqref{eq:pq} is satisfied. Suppose that $A_\omega$ and $F_\omega$ are  defined as in \eqref{eq:A(v_0)omega} and \eqref{eq:F(v)omega} and 
let $v_0=(u_0,h_0,a_0) \in V_\mu$.  
Then there exists $r_0>0$ and a constant $L>0$ such that $\overline{B}_{X_{\gamma,\mu}}(v_0,r_0) \subset V_\mu$ and 
\begin{align*}
 \Vert A_\omega(v_1)w - A_\omega(v_2)w \|_{X_0} & \le L \Vert v_1-v_2\|_{X_{\gamma,\mu}} \Vert w \|_{X_1}, \\
  \Vert F_\omega(v_1) - F_\omega(v_2) \|_{X_0}   & \le L \Vert v_1 - v_2\|_{X_{\gamma,\mu}}.
\end{align*}
for all $v_1,v_2 \in \overline{B}_{X_{\gamma,\mu}}(v_0,r_0)$ and all $w \in X_1$. 
\end{lemma}

\begin{proof}
Choose $r_0>0$ small enough such that  $v_1, v_2 \in  \overline{B}_{X_{\gamma,\mu}}(v_0,r_0) \subset V_\mu$. For $w=(u,h,a) \in X_1$ we then obtain 
\begin{align*}
& \quad \Vert A_\omega(v_1)w - A_\omega(v_2)w \|_{X_0} \\
    &= \Big\Vert  \frac{1}{\rice h_1}[ A_D^H(\nabla u_1,P(h_1,a_1))u + \omega u - \partial_{h}P(h_1,a_1)\nabla h - \partial_{a}P(h_1,a_1)\nabla a] \\ 
    & \quad -   \frac{1}{\rice h_2}[ A_D^H(\nabla u_2,P(h_2,a_2))u + \omega u - \partial_{h}P(h_2,a_2)\nabla h - \partial_{a}P(h_2,a_2)\nabla a] \Big\Vert_{L_q(\Omega;\R^2)} \\
   & \leq L \Vert v_1-v_2 \Vert_{C^1} (\Vert D_iD_j u \Vert_q  + \Vert u \Vert_q) + L  \Vert (h_1,a_1)^T - (h_2,a_2)^T \Vert_\infty (\Vert(\nabla h \Vert_q + \Vert \nabla a \Vert_q) \\
   & \leq L \Vert v_1-v_2 \Vert_{X_{\gamma,\mu}} \Vert w \Vert_{X_1}.  
\end{align*}
To prove the assertion for $F_\omega$ we  start with the convective term $u \nabla u$. H\"older's inequality and the embedding $X_{\gamma,\mu} \hookrightarrow L_s \cap H_s ^1$ for 
$s = qr$ and $s=qr'$ imply
\begin{align*}
    \Vert u_1 \nabla u_1 - u_2 \nabla u_2 \|_{L_q}
    &\le \Vert u_1 - u_2\|_{L_{qr^\prime}} \Vert u_1 \|_{H^1_{qr}} + \Vert u_1 - u_2 \|_{H^{1}_{qr^\prime}} \Vert u_2 \|_{L_{qr}}\\
    &\le 2C r_0 \Vert u_1 - u_2 \|_{X_{\gamma,\mu}^u}.
\end{align*} 
A  similar argument shows that 
\begin{align*}    
\Vert \divergence(u_1 h_1) - \divergence(u_2h_2) \|_{X_0 ^h} &\le C \Vert v_1 - v_2\|_{X_{\gamma,\mu}} \quad \mbox{and}\\
    \Vert \divergence(u_1 a_1) - \divergence(u_2 a_2) \|_{X_0 ^a} &\le C \Vert v_1 - v_2 \|_{X_{\gamma,\mu}}.
\end{align*}
Furthermore, note that   $\tatm$ is constant in $v$ and thus  ${\frac{\tatm}{\rice h_0}}$ is Lipschitz continuous in $v$. 
Concerning $\tocean$, we may assume that $\Uatm = 0$ (otherwise consider $u + \Uatm$). It thus suffices to show  that $ v \mapsto \frac{1}{h}{u \vert u \vert}$ is Lipschitz continuous 
viewed  as a mapping from $V_\mu$ to $L_q$. 
The term $\frac{1}{\rice h}\omega u$ is treated in the same way.

Finally, we consider the terms $S_h$ and $S_a$ defined as in \eqref{def:sh} and \eqref{def:sa}, respectively.  
By assumption, $f \in C^1$ and hence $S_h$ as well as $S_a$ are  Lipschitz continuous in $v$.
\end{proof}

The assertion of Theorem \ref{thm:local} follows hence by the local existence theorem for quasilinear evolution equation as described e.g. in \cite[Thm.~5.1.1]{PS16}.  

\vspace{.2cm}
\section{Proof of Theorem~\ref{thm:global}}\label{sec:global}

Throughout this section, we consider $p,q \in (1,\infty)$ and $\mu \in (\frac{1}{p},1]$ such that \eqref{eq:pq} holds. Moreover, we abbreviate $P(h_\ast,a_\ast)$ by $P_\ast$, i.e.,
$$
P_\ast = p^\ast h_\ast \exp(-c(1-a_\ast)).
$$
We study equilibria in the case that no external forces are present in the momentum equation, i.e., 
$$
 -g \nabla H = \frac{c_1}{h} \vert U_\text{atm} \vert U_\text{atm} = \frac{c_2}{h} \vert U_\text{ocean} - u \vert (U_\text{ocean} - u) = 0,
$$
and neglect 
external freezing and melting effects by setting 
$$S_h = S_a = 0.$$ 
For $A$ as in \eqref{eq:A(v)} and the simplified semilinear right-hand side $F_s$ given by
$$
F_s(v) = \begin{pmatrix}
- u \cdot \nabla u - \ccor(n \times u) \\ - \divergence(u h) \\  -\divergence(u a)
\end{pmatrix},
$$
we prove similarly as in Section~\ref{seclocexofsol} that there is an open set $V \subset V_\mu \subset X_{\gamma,\mu}$ such that
\begin{align}\label{eq:smoothnessaf}
(A,F_s) \in C^1(V,\mathcal{L}(X_1,X_0) \times X_0).
\end{align}
We denote by $\cE \subset V \cap X_1$ the set of equilibrium solutions of
\begin{align}\label{eq:equi}
v^\prime + A(v)v = F_s(v), \quad t>0, \quad v(0)=v_0.    
\end{align}
An equilibrium solution $v \in \cE$ is characterized  by $v \in V \cap X_1$ as well as $A(v) v = F_s(v)$.

For $h_\ast \ge \kappa$ and $a_\ast \ge 0$ 
constant in time and space, 
$v_\ast = (0, h_\ast, a_\ast) \in V \cap X_1$ is an equilibrium solution of \eqref{eq:equi} due to $A(v_\ast)v_\ast = 0 = F_s(v_\ast)$.

To prove  Theorem \ref{thm:global}, we aim to apply the generalized principle of linearized stability, see \cite{PSZ09} or \cite{PS16}. Note that we already verified 
that $v_\ast \in V \cap X_1$ is an equilibrium of \eqref{eq:equi} and that $(A,F_s)$ satisfy \eqref{eq:smoothnessaf}.
Consider next the linearization of \eqref{eq:equi} at $v_\ast$ which reads as
$$
A_0 v = A(v_\ast) v + (A^\prime(v_\ast)v)v_\ast - F_s^\prime(v_\ast)v, v \in X_1.
$$
Computing $A_0 v$ we see first that $A(v_\ast)v$ is given by  
$$
A(v_\ast)v = 
\begin{pmatrix}
\frac{1}{\rho_\text{ice}h_\ast}A_D^H(v_\ast) u + \frac{\partial_h P_\ast}{2 \rho_\text{ice}h_\ast} \nabla h + \frac{\partial_a P_\ast}{2 \rho_\text{ice}h_\ast} \nabla a \\ - d_h \Delta_N h \\ - d_a \Delta_N a
\end{pmatrix},
$$
where 
$$
(A_D^H(v_\ast)u)_i = - \frac{P_\ast}{2 \delta^{1/2}} \sum \limits_{j,k,l=1}^2 \mathbb{S}_{ij}^{kl} \partial_k \partial_l u_j.
$$
Secondly, we deduce that $(A^\prime(v_\ast)v)v_\ast = 0$ for all $v \in X_1$ and that 
$$
F_s^\prime(v_\ast)v = \begin{pmatrix}
- \ccor (n \times u) \\ - h_\ast \divergence(u) \\ - a_\ast \divergence(u)
\end{pmatrix}.
$$
The linearization $A_0$ hence becomes
$$
A_0 v = A(v_\ast)v - F_s^\prime(v_\ast)v = \begin{pmatrix}
\frac{1}{\rho_\text{ice}h_\ast}A_D^H(v_\ast)u + \frac{\partial_h P_\ast}{2 \rho_\text{ice}h_\ast}\nabla h + \frac{\partial_a P_\ast}{2 \rho_\text{ice}h_\ast}\nabla a  - \ccor (n\times u)\\ 
h_\ast \divergence(u) - d_h \Delta_N h \\ a_\ast \divergence(u) - d_a \Delta_N a
\end{pmatrix}.
$$

\begin{lemma}\label{lemma:7.1}
If $v_\ast$ is as above, then there exists $\delta_\ast >0$ such that $\sigma(A_0)\setminus \{0\} \subset \C_+$ holds for all $0<\delta < \delta_\ast$. Furthermore, $0$ is a semi-simple eigenvalue of $A_0$ and $N(A_0)$ has dimension $2$. 
\end{lemma}

\begin{proof}
To locate the spectrum of $A_0$, we test the equation $(\lambda + A_0)v = 0$ by $v=(u,h,a)$ and use integration by parts, which leads to
\begin{align}\label{eq:testingeveq}
    0
    &= \lambda \Vert v \|_{L_2(\Omega)^4}^2 + \frac{1}{\rho_\text{ice}h_\ast}\int_\Omega A_D^H(v_\ast) u \cdot u \,d x + \frac{\partial_h P_\ast}{2 \rho_\text{ice}h_\ast} \int_\Omega \nabla h \cdot u \,d x + \frac{\partial_a P_\ast}{2 \rho_\text{ice}h_\ast} \int_\Omega \nabla a \cdot u \,d x\\
    & \quad + h_\ast \int_\Omega h \divergence(u) \,d x + d_h \Vert \nabla h \|_{L_2(\Omega)^2}^2  + a_\ast \int_\Omega a \divergence(u) \,d x + d_a \Vert \nabla a \|_{L_2(\Omega)^2}^2 \notag.
\end{align}
Thanks to the Dirichlet boundary condition for $u$ and employing \eqref{eq:ellest1} as well as Korn's and 
Poincar\'e's inequality  we deduce that
\begin{align}\label{eq:ahdinvertierbar}
\frac{1}{\rho_\text{ice}h_\ast}\int_\Omega A_D^H(v_\ast) u \cdot u \,d x &= 
- \frac{P_\ast}{2h_\ast \rice \sqrt{\delta}} \int_\Omega  \mathbb{S}_{ij}^{kl}\partial_k\partial_lu_ju_i  \, dx = 
\frac{P_\ast}{2h_\ast \rice \sqrt{\delta}} \int_\Omega  \mathbb{S}_{ij}^{kl} \partial_l u_j\partial_k u_i  \, dx \nonumber \\ 
& =
\frac{P_\ast}{2h_\ast \rice \sqrt{\delta}} \int_\Omega  \Delta^2(\nabla u)  \, dx  \geq 
\frac{P_\ast}{2h_\ast \rice \sqrt{\delta}}\frac{2}{e^2} \int_\Omega  |\eps(u)|^2  \, dx  \geq \frac{C_\ast}{\sqrt{\delta}} \Vert u \Vert_{H^1}^2 
\end{align}
for some constant $C_\ast > 0$ independent of $\delta$ and $u$. 
Now the remaining terms in \eqref{eq:testingeveq} can be absorbed:
First determine $\gamma_h,\gamma_a > 0$ depending in particular on $h_\ast, a_\ast, d_h$ and $d_a$ such that
\begin{align*}
    \frac{\partial_h P_\ast}{2 \rho_\text{ice}h_\ast} \int_\Omega \nabla h \cdot u \,d x +  h_\ast \int_\Omega h \divergence(u) \,d x
    = \Bigl(\frac{\partial_h P_\ast}{2 \rho_\text{ice}h_\ast} - h_\ast\Bigr) \int_\Omega \nabla h \cdot u \,d x
   \ge - \frac{d_h}{2} \Vert \nabla h \|_{L_2(\Omega)^2} - \gamma_h \Vert u \|_{L_2(\Omega)^2},
\end{align*}
and, similarly, such that 
\begin{align*}
    \frac{\partial_a P_\ast}{2 \rho_\text{ice}h_\ast} \int_\Omega \nabla a \cdot u \,d x +  a_\ast \int_\Omega a \divergence(u) \,d x
	    \ge- \frac{d_a}{2} \Vert \nabla a \|_{L_2(\Omega)^2} - \gamma_a \Vert u \|_{L_2(\Omega)^2}.
\end{align*}
Then choose $\delta_\ast > 0$ sufficiently small to ensure that 
$\gamma_h  + \gamma_a  < \frac{C_\ast}{\sqrt{\delta_*}}$.
In particular, this implies that for all $\delta < \delta_\ast$, there exists $C_\delta > 0$ such that
\begin{align}\label{eq:resulteveq}
    0 &\ge \lambda \Vert v \|_{L_2(\Omega)^4}^2 + C_\delta \Bigl(\Vert u \|_{H^1(\Omega)^2}^2 + \Vert \nabla h \|_{L_2(\Omega)^2} + \Vert \nabla a \|_{L_2(\Omega)^2}\Bigr).
\end{align}
The relation in \eqref{eq:resulteveq} can only hold provided that $\lambda$ is real and that $\lambda \le 0$. Hence,   $\sigma(A_0)\setminus \{0\} \subset \C_+$. 
For $\lambda = 0$, we infer that $u=0$ and $h$ as well as $a$ are constant. This implies that $0$ is a semi-simple eigenvalue of $A_0$ and that $N(A_0)$ has dimension $2$. 

\end{proof}

\begin{lemma}\label{lemma:7.2}
Near $v_\ast$, the set of equilibria $\cE$ is a $C^1$-manifold in $X_1$, and the tangent space of $\cE$ at $v_\ast$ is isomorphic to $N(A_0)$.
\end{lemma}

\begin{proof}
Consider equilibria $v \in V \cap X_1$ such that $\Vert v - v_\ast \|_{X_{\gamma,\mu}} < r$ for given $r > 0$. 
The resulting equation for such $v$ is
$$
0= \begin{pmatrix}
A_D^H(v)u + \nabla \frac{P(v)}{2} + \rho_\text{ice} h \, u \cdot \nabla u - \ccor (n \times u) \\ \divergence(u h) - d_h \Delta_N h \\ \divergence(u a) - d_a \Delta_N a
\end{pmatrix}.
$$
We set the constants
\begin{equation}\label{eq:ChCa}
C_h := \frac{p^\ast \exp(-c(1-a_\ast))}{2 h_\ast} \mbox{ and } C_a := \begin{cases} \frac{c p^\ast h_\ast \exp(-c(1-a_\ast))}{2 a_\ast}, & \text{if } a_\ast >0, \\
1, & \text{if } a_\ast = 0,  
\end{cases}
\end{equation}
and test the above equation with $(u, C_h h, C_a a)$ to obtain
\begin{align}\label{eq:testingcloseequil}
    0
    &= \int_\Omega A_D^H(v)u \cdot u \,d x  + \int_\Omega  \frac{\nabla P(v)}{2} \cdot u \,d x + \int_\Omega \rho_\text{ice} h (u \cdot \nabla u) \cdot u \,d x + C_h \int_\Omega \divergence(u h) h \,d x\\ 
    &\quad + d_h C_h \Vert \nabla h\Vert_{L_2(\Omega)^2}^2 + C_a \int_\Omega \divergence(u a) a \,d x + d_a C_a \Vert \nabla a \Vert_{L_2(\Omega)^2}^2 \notag.
\end{align}
Using the symmetry of $\mathbb{S}$, the estimate $P(v)\ge P_\ast \kappa \exp(-c(1-\alpha))$, the estimate $\triangle^2_\delta(\eps) \le C_e r$ and Korn's and Poincare's inequalities, the first term on the right-hand-side satisfies
\begin{align*}
\int_\Omega A_D^H(v)u \cdot u \,d x & = - \int_\Omega \mathrm{div} \Bigl( \frac{P}{2} \frac{\mathbb{S}\eps}{\triangle_\delta(\eps)} \Bigr) \cdot u \, dx  = \int_\Omega  \frac{P}{2} \frac{\eps^T \mathbb{S} \eps}{\triangle_\delta(\eps)} \, dx 
\ge \frac{C_V}{\sqrt{\delta + C_e r}}
 \Vert u \Vert_{H_1(\Omega)}^2
\end{align*}
for some constant $C_V >0$ independent of $\delta,r$ and $v$. 
We show how terms without sign in \eqref{eq:testingcloseequil} can now be absorbed.  We first discuss the case $a_\ast \neq 0$ and remark on the case $a_\ast = 0$ below. 
First note that using $\Vert v - v_\ast \|_{X_{\gamma,\mu}} < r$ and any bound on $r>0$,
$$
\int_\Omega \rho_\text{ice} h (u \cdot \nabla u) \cdot u \,d x \leq C_\ast r \Vert u \|_{H^1(\Omega)},
$$
for a suitable constant $C_\ast > 0$ that is independent of $\delta, r$ and $v$. 
Secondly, we calculate 
\begin{align}\label{eq:absorb1}
\int_\Omega  \frac{\nabla P(v)}{2}\cdot u \,d x & = \int_\Omega \Bigl(\frac{\partial_h P}{2} - C_h h\Bigr) \nabla h \cdot u  \,d x + C_h \int_\Omega h \nabla h \cdot u \,d x \\ \notag 
&+ \int_\Omega \Bigl(\frac{\partial_a P}{2} - C_a a\Bigr) \nabla a \cdot u \,d x + C_a \int_\Omega a \nabla a \cdot u \,d x
\end{align}
and use that part of this expression cancels with the terms
$$
C_h \int_\Omega \divergence(u h)h \,d x = - C_h \int_\Omega h (\nabla h \cdot u) \,d x \mbox{ as well as } C_a \int_\Omega \divergence(u a)a \,d x = - C_a \int_\Omega a (\nabla a \cdot u) \,d x
$$
in \eqref{eq:testingcloseequil}. It remains to check that due to the particular choice of $C_h,C_a$ in \eqref{eq:ChCa}, we find that for a (possibly increased) constant $C_\ast > 0$, 
$$
\Vert \frac{\partial_h P}{2} - C_h h \Vert_\infty \leq \frac{p^\ast}{2} \Bigl( \Vert \exp(ca) - \exp(ca_\ast) \|_\infty + h_\ast \exp(-c(1-a_\ast)) \Vert h - h_\ast \|_\infty\Bigr) \leq C_\ast r
$$
and similarly 
$$ \Vert \frac{\partial_a P}{2} - C_a a \Vert_\infty \leq C_\ast r,$$
and hence the terms 
$$
\int_\Omega \Bigl(\frac{\partial_h P}{2} - C_h h\Bigr) \nabla h \cdot u  \,d x \geq - 
C_\ast r\Bigl(\Vert \nabla h \|_{L_2(\Omega)^2}^2 + \Vert u \|_{L_2(\Omega)^2}^2\Bigr)$$ 
and 
$$\int_\Omega \Bigl(\frac{\partial_a P}{2} - C_a a\Bigr) \nabla a \cdot u \,d x
\geq - C_\ast r\Bigl(\Vert \nabla a \|_{L_2(\Omega)^2}^2 + \Vert u \|_{L_2(\Omega)^2}^2\Bigr)$$
are controlled. 


In summary, inserting the above estimates into equation  \eqref{eq:testingcloseequil}, we conclude that
\begin{align*}
 0 \ge \Bigl(\frac{C_V}{\sqrt{\delta + C_e r}} -  C_\ast r\Bigr)\Vert u \|_{H^1(\Omega)^2}^2 + \Bigl(d_h C_h - C_\ast r\Bigr)\Vert \nabla h\|_{L_2(\Omega)^2}^2
    + \Bigl(d_a C_a - C_\ast r\Bigr)\Vert \nabla a\|_{L_2(\Omega)^2}^2.
\end{align*}
Hence, if  $r > 0$ is sufficiently small, then \eqref{eq:testingcloseequil} implies
\begin{align}\label{eq:resulttestingcloseequil}
    0
    &\ge \Vert u \|_{H^1(\Omega)^2}^2 + \Vert \nabla h\|_{L_2(\Omega)^2}^2 + \Vert \nabla a\|_{L_2(\Omega)^2}^2.
\end{align}
This shows that for $v=(u,h,a) \in V_\ast$ with $\Vert v - v_\ast \|_{X_{\gamma,\mu}} < r$, we have  $u=0$ and $h$ as well as $a$ must be constant. 
 In particular, $\cE = N(A_0)$ is valid in a neighborhood of $v_\ast$.\\
The case $a_\ast = 0$ can be included by a slight adjustment of the argument. Replace \eqref{eq:absorb1} by
\begin{align*}
\int_\Omega  \frac{\nabla P(v)}{2}\cdot u \,d x = \int_\Omega \Bigl(\frac{\partial_h P}{2} - C_h h\Bigr) \nabla h \cdot u  \,d x + C_h \int_\Omega h \nabla h \cdot u \,d x + \int_\Omega \frac{\partial_a P}{2}  \nabla a \cdot u \,d x,
\end{align*}
and directly estimate
$$ 
 \int_\Omega \frac{\partial_a P}{2}  \nabla a \cdot u \,d x  \geq - C_* \Vert a \Vert_\infty \int_\Omega \nabla a \cdot u \,dx \geq - C_* r \Bigl(\Vert \nabla a \|_{L_2(\Omega)^2}^2 + \Vert u \|_{L_2(\Omega)^2}^2\Bigr)
$$
as well as 
$$
\int_\Omega \divergence(u a)a \,d x = -  \int_\Omega a (\nabla a \cdot u) \,d x \geq - C_* r \Bigl(\Vert \nabla a \|_{L_2(\Omega)^2}^2 + \Vert u \|_{L_2(\Omega)^2}^2\Bigr),
$$
to conclude as before.  
\end{proof}

\begin{lemma}\label{lemma:7.3}
For $v_\ast$ as above, $A(v_\ast)$  has the property of maximal $L_s$-regularity on $L_r(\Omega;\R^2)$.
\end{lemma}

\begin{proof}
We already know from Theorem \ref{thm:mr} that there exists $\omega_0 \in \R$ such that $A^H_D(v_\ast) + \omega$ has the maximal $L_s$-regularity on $L_r(\Omega;\R^2)$ for all $\omega > \omega_0$.
Considering the eigenvalue equation for $A^H_D(v_\ast)$ it follows by \eqref{eq:ahdinvertierbar} that 
$$ 
0 = \lambda \Vert u \Vert^2 + \frac{1}{\rice h_\ast} \int_\Omega A^H_D(v_\ast)u \cdot u \, dx \geq  \lambda \Vert u \Vert^2 + C \Vert u \Vert_{H^1}^2.
$$
Thus $s(-A^H_{D,2}(v_\ast)) < 0$ and $A^H_{D,2}(v_\ast)$ is invertible in $L_2(\Omega)^2$.  Due to compact embeddings, $A^H_D(v_\ast)$ has compact resolvent and hence the spectrum of $A^H_D(v_0)$ is 
$r$-independent and we see that $s(-A^H_{D,2}(v_\ast))= s(-A^H_D(v_\ast))<0$. It can be shown that $\omega_0$ can be chosen to be equal to the spectral 
bound $s(-A^H_D(v_\ast))$ of $A^H_D(v_\ast)$, i.e., $\omega_0=s(-A^H_D(v_\ast))$, which implies that $A^H_D(v_\ast)$ has the maximal $L_s$-regularity on $L_r(\Omega;\R^2)$. 
The triangular structure of $A(v_\ast)$ implies that this the latter property holds also for $A(v_\ast)$.  
\end{proof}

Summarizing we see that Lemmas \ref{lemma:7.1}, \ref{lemma:7.2} and \ref{lemma:7.3} imply that the assumptions of the principle of  linearized stability described as in 
\cite{PSZ09} or \cite{PS16} are fulfilled.  The assertion of Theorem \ref{thm:global} follows thus by this principle.

\end{document}